\documentclass[12pt,a4paper]{amsart}
\usepackage{amssymb,enumerate}

\newcommand{\CC}{{\mathbb{C}}}
\newcommand{\FF}{{\mathbb{F}}}
\newcommand{\GG}{{\mathbb{G}}}
\newcommand{\NN}{{\mathbb{N}}}
\newcommand{\ZZ}{{\mathbb{Z}}}

\newcommand{\ba} {\mathbf a}
\newcommand{\bb} {\mathbf b}
\newcommand{\bc} {\mathbf c}
\newcommand{\bd} {\mathbf d}
\newcommand{\bt} {\mathbf t}
\newcommand{\bu} {\mathbf u}
\newcommand{\bw} {\mathbf w}
\newcommand{\bG} {\mathbf G}
\newcommand{\bP} {\mathbf P}
\newcommand{\bQ} {\mathbf Q}
\newcommand{\bR} {\mathbf R}
\newcommand{\bS} {\mathbf S}

\newcommand{\cA} {\mathcal A}
\newcommand{\cC} {\mathcal C}
\newcommand{\cF} {\mathcal F}
\newcommand{\cR} {\mathcal R}
\newcommand{\cW} {\mathcal W}
\newcommand{\cX} {\mathcal X}

\newcommand{\fc} {\mathfrak c}
\newcommand{\fS} {\mathfrak S}

\newcommand{\Aut}{{{\operatorname{Aut}}}}
\newcommand{\Hom}{{{\operatorname{Hom}}}}
\newcommand{\IBr}{{{\operatorname{IBr}}}}
\newcommand{\Inn}{{{\operatorname{Inn}}}}
\newcommand{\Irr}{{{\operatorname{Irr}}}}
\newcommand{\Out}{{{\operatorname{Out}}}}
\newcommand{\Syl}{{{\operatorname{Syl}}}}

\newcommand{\GL}{\operatorname{GL}}
\newcommand{\SL}{\operatorname{SL}}
\newcommand{\PGL}{\operatorname{PGL}}
\newcommand{\CSp}{{{\operatorname{CSp}}}}
\newcommand{\Sp}{{{\operatorname{Sp}}}}
\newcommand{\OO}{\operatorname{O}}
\newcommand{\GO}{\operatorname{GO}}
\newcommand{\SO}{\operatorname{SO}}
\newcommand{\U}{\operatorname{U}}

\let\vhi=\varphi
\let\ga=\gamma

\newtheorem{thm}{Theorem}[section]
\newtheorem{lem}[thm]{Lemma}

\newtheorem{prop}[thm]{Proposition}

\newtheorem{thmA}{Theorem}

\theoremstyle{definition}
\newtheorem{rem}[thm]{Remark}

\newtheorem{exmp}[thm]{Example}

\begin{document}

\title{Weights for compact connected Lie groups}

\author{Radha Kessar}
\address{Department of Mathematics, University of Manchester,
  United Kingdom}
\email{radha.kessar@manchester.ac.uk}

\author{Gunter Malle}
\address{FB Mathematik, TU Kaiserslautern, Postfach 3049,
  67653 Kaisers\-lautern, Germany.}
\email{malle@mathematik.uni-kl.de}

\author{Jason Semeraro}
\address{Heilbronn Institute for Mathematical Research, Department of
  Mathematics, University of Leicester, United Kingdom}
\email{jpgs1@leicester.ac.uk}

\thanks{We thank Bob Oliver  for illuminating  conversations. The second author gratefully acknowledges support by the DFG--- Project-ID 286237555 -- TRR 195. }

\begin{abstract}
Let $\ell$ be a prime. If $\bG$ is a compact connected Lie group, or a connected
reductive algebraic group in characteristic different from $\ell$, and $\ell$ is
a good prime for $\bG$, we show that the number of weights of the $\ell$-fusion
system of $\bG$ is equal to the number of irreducible characters of its Weyl
group. The proof relies on the classification of $\ell$-stubborn subgroups in
compact Lie groups.
\end{abstract}



\date{\today}

\maketitle


\section{Introduction}

In this paper we prove the following identity, first conjectured in \cite{S23}:

\begin{thmA}   \label{thm:compweights2}
 Suppose that $\bG$ is a compact connected Lie group. Let $\ell$ be a prime
 which is good for $\bG$, $S$ a Sylow $\ell$-subgroup of $\bG$,
 $\cF= \cF_S(\bG)$ the corresponding saturated fusion system and $W$ the Weyl
 group of $\bG$. Then,
 \[\bw(\cF) = |\Irr(W)|. \]
\end{thmA}

Here $\bw(\cF)$ denotes the number of weights of the $\ell$-fusion system
attached to $\bG$ (see Section \ref{sec:prelim} for precise definitions)  and $\Irr(W)$  denotes the  
set  of  ordinary  irreducible characters of $W$. We recall the  definition of good primes  later in this  introduction.  While
the stated identity is rather easy to see for primes not dividing $|W|$, it
seems totally mysterious for the other primes.
\medskip

The origin of our interest in this theorem is Alperin's Weight Conjecture (AWC)
in finite group representation theory which relates the number
of simple modules of a block of a finite group algebra over a field of
characteristic $\ell$ to the number of simple, projective modules of smaller
``local" groups. Indeed, via some deep representation theory,  it can be shown  that    
 for the principal block of the group algebra of a
finite group of Lie type whose defining characteristic is different
from~$\ell$, and under some extra conditions,  AWC is equivalent to an equation
of the above type (see \cite[Prop.~4.1]{KMS}). In \cite[Thm~1]{KMS}, we proved
this equality for finite groups of Lie type (under the relevant conditions)  and made the surprising discovery
that it continues to hold in a broader setting than that of finite groups, in
particular for certain exotic fusion systems on finite $\ell$-groups arising
from homotopy fixed point spaces of connected $\ell$-compact groups. We view
Theorem~\ref{thm:compweights2}, where the underlying $\ell$-group is an infinite
group, as additional evidence for AWC and a further indication that it is not
restricted to the world of finite groups. We also note that
Theorem~\ref{thm:compweights2} yields the analogous identity for $\ell$-weights
of fusion systems of connected reductive algebraic groups in characteristic
different from $\ell$:

\begin{thmA}   \label{thm:compweights3}
 Let $\ell$ and $p$ be distinct primes. Suppose that $\bG$ is a connected
 reductive algebraic group over $\overline{\FF}_p$ such that $\ell$ is good
 for $\bG$, let $S$ be a Sylow $\ell$-subgroup of $\bG$, $\cF=\cF_S(\bG)$
 the corresponding saturated fusion system and $W$ the Weyl group of~$\bG$.
 Then,
 \[\bw(\cF) = |\Irr(W)|. \]
\end{thmA}

The proof of Theorem~\ref{thm:compweights2} proceeds by transferring from the
discrete setting to the continuous one in which the role of centric, radical
subgroups in fusion systems is taken over by $\ell$-stubborn subgroups (in the
sense of Jackowski--McClure--Oliver \cite{JMO92-I}) of compact Lie groups. This
leads to a straightforward reduction to the case that the underlying compact
group is simple, at which stage we are able to invoke known classifications of
$\ell$-stubborn subgroups \cite{O94}, \cite{JMO95}, \cite{Vi01}. The proof for
the classical groups involves counting arguments similar to those used for
\cite[Thm~1]{KMS} which originated in the work of Alperin--Fong  \cite{AF90}.
\medskip

Recall that a prime $\ell$ is \emph{bad} for a root system $\Phi$ if $\ZZ\Phi/\ZZ\Psi$
has $\ell$-torsion for some closed subsystem $\Psi\subset\Phi$ (see e.g.
\cite[Def.~B.24]{MT}). Thus, it is bad for $\Phi$ if it is so for one of its
indecomposable summands. The bad primes for indecomposable root systems are:
none for type $A_n$, $\ell=2$ for $B_n,C_n$ and $D_n$, $\ell=2,3$ for
$G_2,F_4,E_6,E_7$, and $\ell=2,3,5$ for $E_8$ (see \cite[Tab.~14.1]{MT}). In
particular, all primes $\ell\ge7$ are good (that is, non bad)
for all types. In Proposition~\ref{prop:badfail} we show that the conclusion of
Theorem~\ref{thm:compweights2} does not hold for the compact symplectic groups
at the bad prime $\ell=2$, and in Proposition~\ref{prop:F4} that it does not
hold for $\bG=F_4$ at the bad prime $\ell=3$.
Still in both cases we obtain $\bw(\cF)\le |\Irr(W)|$, as conjectured in
\cite{S23}. In Example~\ref{exmp:finite} we observe that for finite reductive
groups in bad characteristic we may have $\bw(\cF) > |\Irr(W)|$.
\medskip

Recall that by results of Friedlander (see \cite[Thm~3.1]{BMO} and
Section~\ref{subsec:red}), the $\ell$-completed classifying space (and
therefore the $\ell$-fusion system) of a
finite group of Lie type in non-describing characteristic can be recovered by
taking homotopy fixed points under suitable unstable Adams operations on the
$\ell$-completed classifying space of the corresponding compact Lie group. It
would be desirable to obtain a direct connection between the weight equation
demonstrated in Theorems~\ref{thm:compweights2} and~\ref{thm:compweights3} and
that shown in ~\cite[Thm~1]{KMS}, but as yet we do not see such a connection
either via unstable Adams operations or via Frobenius morphisms on algebraic
groups. In particular, we do not see how to deduce the equation in the finite
setting from the infinite one or vice-versa. Note that our results show that
the number of weights for compact groups does not depend on the isogeny type,
while for example the two isogenous groups $\SL_3(4)$ and $\PGL_3(4)$ have
five respectively three 3-weights. In another direction, we believe
that the purview of Theorem~\ref{thm:compweights2} could be expanded to also
cover fusion systems of $\ell$-compact groups as defined in
\cite[Sec.~10]{BLO07}.
\medskip

The paper is organised as follows. In Section~\ref{sec:prelim}, we recall the
relevant background material. Section~\ref{sec:proof} contains the proofs of
Theorems~\ref{thm:compweights2} and~\ref{thm:compweights3} and in
Section~\ref{sec:badprimes}, we present our calculations for bad primes.

\section{Preliminaries}   \label{sec:prelim}

In this section we recall some aspects of the theory of saturated fusion
systems associated to compact Lie groups as set up by Broto, Levi and Oliver
in \cite{BLO07}.
For subgroups $Q, R$ of a group $G$, $\Hom_G(Q,R)$ consists of the group
homomorphisms from $Q$ to $R$ induced by conjugation by elements of $G$,
$\Aut_G(Q):=\Hom_G(Q, R)\cong N_G(Q)/C_G(Q)$ is the group of $G$-automorphisms
of $Q$, and $\Out_G(Q) = \Aut_G(Q)/\Inn(Q) \cong N_G(Q)/QC_G(Q)$ is the
corresponding group of outer automorphisms. For $H \leq G$ denote by $\cF_H(G)$
the category with objects the subgroups of $H$, in which the $\cF$-morphisms
from $Q$ to $R$ are the elements of $\Hom_G(Q,R)$ for $Q, R \leq H$, and
composition of morphisms is the usual composition of maps. Let $\ell$ be a
prime number.

\subsection{Fusion systems on discrete $\ell$-toral groups.}
A \emph{discrete $\ell$-toral group} is a group $P$ with normal subgroup
$P^\circ$ such that $P^\circ$ is isomorphic to a finite product of copies of
$\ZZ/\ell^\infty$ and $P/P^\circ$ is a finite $\ell$-group. In particular, any
finite $\ell$-group is a discrete $\ell$-toral group. Here $P^\circ$ is
characterised as the subset of infinitely divisible elements of $P$ as well as
the minimal subgroup of finite index of $P$. In particular, $P^\circ$ is
characteristic in $P$. It is called the \emph{identity component of $P$} and
$P$ is \emph{connected} if $P=P^\circ$.
\medskip
 
Let $S$ be a discrete $\ell$-toral group and let $\cF$ be a saturated fusion
system on $S$ as defined in \cite[Sec.~2]{BLO07}. Following \cite{S23}, the
the number of weights of $\cF$ is defined by 
\[ \bw(\cF) := \sum_Q z(\Out_\cF(Q)), \]
where $Q$ runs over a set of representatives of $\cF$-conjugacy classes of
$\cF$-centric, $\cF$-radical subgroups of $S$,
$\Out_\cF(Q) = \Aut_\cF(Q) /\Inn(Q)$ is the group of $\cF$-outer automorphisms
of $Q$, and $z(\Out_\cF(Q))$ is the number of ordinary irreducible characters
of $\Out_\cF(Q)$ characters of zero $\ell$-defect. Note that by
\cite[Cor.~3.5, Defn.~2.2]{BLO07}, $S$ has only finitely many classes of
$\cF$-centric, $\cF$-radical subgroups and moreover, the $\cF$-outer
automorphism group of any $\cF$-centric, $\cF$-radical subgroup of $S$ is
finite. Thus, $\bw(\cF) $ is well defined. Also, note that if $G$ is a finite
group, $S$ is a Sylow $\ell$-subgroup of $G$ and $\cF=\cF_S(G)$ is the
corresponding saturated fusion system, then $\bw(\cF)$ is the number of Alperin
weights associated to the principal block of $kG$, $k$ an algebraically closed
field of characteristic $\ell$ (see \cite[Thm~8.14.4]{Li18-II}).
\medskip

Following \cite[Sec.~8]{BLO07}, we say a group $G$ ``has Sylow
$\ell$-subgroups" if there is a discrete $\ell$-toral subgroup $S\leq G$ which
contains all discrete $\ell$-toral subgroups of $G$ up to conjugacy.
Such a subgroup, if it exists, is called a \emph{Sylow $\ell$-subgroup of $G$}.
Note that the set of Sylow $\ell$-subgroups of $G$ is a single
$G$-conjugacy class. An \emph{$\ell$-centric subgroup} of $G$ is a discrete
$\ell$-toral subgroup $ P \leq G$ such that $C_G(P)$ has a Sylow
$\ell$-subgroup, which is $Z(P)$ (and unique). Equivalently, a discrete
$\ell$-toral $P \leq G $ is $\ell$-centric in $G$ if $C_G(P)/Z(P)$ contains no
elements of order $\ell$.

\subsection{Fusion systems of compact Lie groups.}
Let $\bG$ be a compact Lie group. The following definitions and results are
taken from \cite[Sec.~9]{BLO07}.

\begin{itemize}
\item An \emph{$\ell$-toral group} is a compact Lie group whose identity
component is a torus and whose group of components is an $\ell$-group. If $P$
is a discrete $\ell$-toral subgroup of $\bG$, then the closure $\bar P$ of $P$
in $\bG$ is an $\ell$-toral group.
\item For any $\ell$-toral group $\bP$, $\Syl_\ell(\bP)$ denotes the set of all
discrete $\ell$-toral subgroups $P$ of $\bP$ such that $\bP^\circ.P=\bP$, where
$\bP^\circ$ is the identity component of $\bP$ and $P$ contains all
$\ell$-power torsion of $\bP$. The elements of $\Syl_\ell(\bP)$ form a single
$\bP$-conjugacy class.
\item We denote by $\overline \Syl_\ell(\bG)$ the set of all $\ell$-toral
 subgroups $\bS$ of $\bG$ such that the identity component $\bS^\circ$ is a
 maximal torus of $\bG$ and $\bS/\bS^\circ $ is a Sylow $\ell$-subgroup of
 $N_\bG(\bS^\circ)/\bS^\circ$; $ \Syl_\ell(\bG)$ is the set of all discrete
 $\ell$-toral subgroups $P$ of $\bG$ such that $\bar P\in\overline\Syl_\ell(\bG)$ and $P\in \Syl_\ell(\bar P)$. Any two elements of
 $\overline \Syl_\ell(\bG)$ are $\bG$-conjugate and each $\ell$-toral subgroup
 of $\bG$ is contained in an element of $\overline\Syl_\ell(\bG)$. Any two
 elements of $\Syl_\ell(\bG)$ are $\bG$-conjugate and each discrete
 $\ell$-toral subgroup of $\bG$ is contained in an element of $\Syl_\ell(\bG)$
 (see \cite[Prop.~9.3]{BLO07}). In particular, the elements of $\Syl_\ell(\bG)$
 are Sylow $\ell$-subgroups of $\bG$.
\end{itemize}

If $S \leq\bG$ is a Sylow $\ell$-subgroup then by \cite[Lemma~9.5]{BLO07},
$\cF:= \cF_{S}(\bG)$ is a saturated fusion system on $S$.

\begin{itemize}
\item An \emph{$\ell$-stubborn subgroup} of $\bG$ is an $\ell$-toral subgroup
 $\bP$ of $\bG$ such that $N_\bG(\bP)/\bP$ is a finite group which satisfies
 $O_\ell(N_\bG(\bP)/\bP)=1 $.
\item An $\ell$-toral subgroup $\bP \leq \bG$ is called \emph{$\ell$-centric}
 if $Z(\bP)\in \overline \Syl_\ell(C_{\bG}(\bP))$.
\item A discrete $\ell$-toral subgroup $P$ of a compact Lie group $\bG$ is said
 to be \emph{snugly embedded} in $\bG$ if $P \in \Syl_\ell(\bar P)$.
\end{itemize} 

The following well-known result is the first step in the proof of
Theorem~\ref{thm:compweights2} (see also \cite[Sec.~4]{BCGLS}).

\begin{lem}   \label{lem:radicalistubborn}
 Let $\bG$ be a compact Lie group with Sylow $\ell$-subgroup $S$ and let
 $\cF= \cF_S(\bG)$.
 \begin{enumerate}[\rm(a)]
  \item For any $\ell$-stubborn subgroup $\bP$ of $\bG$,
   $C_{\bG^\circ}(\bP)\leq Z(\bP)$ and if $\bG /\bG^\circ$ is an $\ell$-group,
   then $C_\bG(\bP)\leq Z(\bP)$, and consequently 
   $\Out_\bG(\bP)\cong N_\bG(\bP) /\bP$.
  \item Suppose that $P\leq S $ is $\cF$-centric and $\cF$-radical. Then
   $\bP:=\bar P$ is an $\ell$-stubborn subgroup of $\bG$ and
   $\Out_\cF(P) \cong \Out_\bG(\bP)$.
  \item The map $P\mapsto\bar P$ induces an injective map from the set of
   $\cF$-classes of $\cF$-centric, $\cF$-radical subgroups of $S$ to the set of
   $\bG$-classes of $\ell$-stubborn subgroups of $\bG$.
  \item If $\bG /\bG^\circ$ is an $\ell$-group or an $\ell'$-group, then the map
   in {\rm(c)} is a bijection. 
 \item Suppose that $\bG/\bG^\circ $ is an $\ell$-group. Let $\bar \cF$ be the
  orbit category of $\cF$ and $\bar \cF^{cr}$ be the full subcategory whose
  objects are $\cF$-centric, $\cF$-radical subgroups of $S$. Let $\cR_\ell(\bG)$
  be the full subcategory of the orbit category of $\bG$ whose objects are the
  $\ell$-toral subgroups of $\bG$. There is an equivalence of categories
  $\bar \cF^{cr} \simeq \cR_\ell(\bG)$ which for $\cF$-centric, $\cF$-radical
  subgroups $P$, $Q$ of $S$ sends $P$ to $\bG/\bar P$ and which sends the
  $\bar \cF^{cr}$-morphism from $P$ to $Q$ induced by conjugation by $g\in\bG$
  to the $\cR_\ell(\bG)$-morphism from $\bG/\bar P$ to $\bG/\bar Q$ defined by
  $x\bar P \to x g^{-1}\bar Q$, $x \in \bG$.
 \end{enumerate}
\end{lem} 

\begin{proof}
Part (a) is Lemma~7 of \cite{O94}.
Let $P \leq S$ be $\cF$-centric and radical. By \cite[Lemma~3.2 and
Cor.~3.5]{BLO07}, $P = P^\bullet$, where $P^\bullet$ is as defined in Section~3
of \cite{BLO07} and hence by \cite[Lemma~9.9]{BLO07}, $P$ is snugly embedded
in~$\bG$. By Lemma~9.4 of \cite{BLO07} and its proof we have
$\Out_\cF(P) =\Out_\bG(P) \cong \Out_\bG(\bP)$.
Since $P$ is $\cF$-centric, by \cite[Lemma~8.4]{BLO07}, $P$ is $\ell$-centric in
$\bG$ and hence by \cite[Lemma~9.6(a),(b)]{BLO07}, $\bP$ is $\ell$-centric
in~$\bG$, $N_\bG(\bP)/\bP$ is finite and
$\bP C_\bG(\bP)/\bP\cong C_\bG(\bP)/Z(\bP)$ is finite of order prime to $\ell$.
On the other hand, since $\Out_\cF(P) \cong \Out_\bG(\bP) \cong N_\bG(\bP)/\bP C_\bG(\bP)$ and $P$ is $\cF$-radical, $O_\ell(N_\bG(\bP)/\bP C_\bG(\bP))= 1$,
hence $O_\ell(N_\bG(\bP)/\bP) \leq \bP C_\bG(\bP) /\bP $. Since the latter
is an $\ell'$-group, $O_\ell(N_\bG(\bP)/\bP) =1$, showing that $\bP$ is
$\ell$-stubborn. This proves (b).

The assignment $P\mapsto \bP:=\bar P$ induces a map from the set of
$\bG$-classes of discrete $\ell$-toral subgroups of $\bG$ which are snugly
embedded in $\bG$ to the set of $\bG$-classes of $\ell$-toral subgroups
of~$\bG$. By
\cite[Prop.~9.3] {BLO07}, this map is a bijection with the inverse being the map
which sends the $\bG$-class of an $\ell$-toral subgroup $\bP$ of $\bG$ to the
$\bG$-class of $P$, where $P \in \Syl_\ell(\bP)$. 
Now (c) follows from (b) and the fact that by the definition of $\cF$,
$P,P'\leq S$ are $\cF$-conjugate if and only if $P$ and $P'$ are $\bG$-conjugate.

Now suppose that $P\leq\bG$ is a snugly embedded subgroup such that
$\bP:= \bar P \leq\bG$ is $\ell$-stubborn. From the above, we see that in order
to prove (d) it suffices to show that $P$ is $\cF$-centric and $\cF$-radical
provided that $\bG/\bG^\circ$ is an $\ell$-group or an $\ell'$-group. For this,
note that as above since $P$ is snugly embedded in $\bG$,
$\Out_\bG(P)\cong\Out_\bG(\bP)$ by \cite[Lemma~9.4]{BLO07}. Suppose first that
$\bG/\bG^\circ$ is an $\ell$-group. By~(a), $C_\bG(\bP) =Z(\bP)$, hence
$\Out_\bG(P) = \Out_\cF(\bP) \cong N_\bG(\bP)/\bP$ and by hypothesis
$O_\ell(N_\bG(\bP)/\bP)=1$. This shows that $P$ is $\cF$-radical. Also, since
$C_\bG(\bP)\leq Z(\bP)$, $\bP$ is $\ell$-centric in $\bG$. Hence by
\cite[Lemma~9.6]{BLO07}, $P$ is $\ell$-centric in $\bG$ and consequently by
\cite[Lemma~8.4]{BLO07}, $P$ is $\cF$-centric. Now suppose that $\bG/\bG^\circ$
is an $\ell'$-group. We claim that $\bP\leq\bG^\circ$. Indeed, we have
$\bP = \bP^\circ.P$ and $\bP^\circ\leq\bG^\circ$. On the other hand, every
element of $P$ has finite order a power of $\ell$ and $\bG/\bG^\circ$ is an
$\ell'$-group, hence $P\leq\bG^\circ$. This proves the claim. Since $\bG^\circ$
is normal in $\bG$, $\bP$ is $\ell$-stubborn in $\bG^\circ$, hence again by
\cite[Lemma~7]{O94}, $C_{\bG^\circ}(\bP) =Z(\bP)$ and hence $C_\bG(\bP)/Z(\bP)
 = C_\bG(\bP)/C_{\bG^\circ}(\bP)\leq \bG/\bG^\circ$ is a finite $\ell'$-group.
This shows that $Z(\bP)$ contains all torsion elements of $\ell$-power order of
$C_\bG(\bP)$ and hence that $\bP$ is $\ell$-centric in $\bG$. Now it follows as
in the previous case that $P$ is $\cF$-centric. This completes the proof of (d).

Now suppose that $P, Q \leq S$ are $\cF$-centric and $\cF$-radical. As
above, $P$ and $Q$ are snugly embedded in $\bG$. By
\cite[Lemma~9.4(c)]{BLO07}, there is a bijection 
\[  \Inn(Q) \backslash \Hom_\bG(P,Q)
  \longrightarrow \Inn(\bQ)\backslash \Hom_\bG(\bP,\bQ) \]
induced by the canonical map from $\Hom_\bG(P,Q)$ to $\Hom_\bG(\bP,\bQ)$.
On the one hand, $\Hom_{\bar\cF^{cr}}(P,Q)$ may be identified with
$\Inn(Q)\backslash\Hom_\bG(P,Q)$. On the other hand, by (a), $C_\bG(\bar P) \leq Z(\bar P)$ from which it follows that 
$\Inn(\bQ)\backslash \Hom_\bG(\bP,\bQ)$ may be identified with 
$\mathrm{Mor}_{\cR_\ell(\bG)}(\bG/\bP,\bG/\bQ)$. Now~(e) follows by~(d).
\end{proof}

As an immediate consequence of the previous lemma we get:

\begin{lem}   \label{lem:crtostubborn}
 Suppose that $\bG$ is a compact connected Lie group, $\ell$ a prime,
 $S\in\Syl_\ell(\bG)$ a Sylow $\ell$-subgroup of $\bG$, $\cF= \cF_S(\bG)$. Then
 $$\bw(\cF) = \sum_\bP z(N_\bG(\bP)/\bP) $$
 where $\bP$ runs over a set of representatives of $\bG$-classes of
 $\ell$-stubborn subgroups of $\bG$ and $z(N_\bG(\bP)/\bP)$ is the number
 of irreducible characters of $N_\bG(\bP)/\bP$ of zero $\ell$-defect.
\end{lem}

\subsection{Fusion systems of connected reductive algebraic groups.}\label{subsec:red}
Let $p,\ell$ be distinct prime numbers and let $\bG$ be a connected reductive
algebraic group over $\overline{\FF}_p$. Then since $\bG$ has a finite
dimensional faithful representation over $\overline{\FF}_p$, $\bG$ is a linear
torsion group in characteristic $p$ (in the sense of \cite[Sec.~8]{BLO07}).
Hence, by \cite[Thm~8.10]{BLO07}, $\bG$ has a Sylow $\ell$-subgroup $S$ and
$\cF_S(\bG)$ is a saturated fusion system. Also, $\bG$ is the group of
$\overline{\FF}_p$-points of a connected split reductive algebraic group scheme
$\GG$ over $\ZZ$. By a theorem of Friedlander--Mislin \cite[Thm~1.4]{FM84},
there is a homotopy equivalence of $\ell$-completed classifying spaces
\[ B\GG(\CC)^\wedge_\ell \simeq (B\bG)^\wedge_\ell. \]
Further, by the Cartan--Chevalley--Iwasawa--Malcev--Mostow theorem (see
\cite[Chap.~XV, Thm~3.1]{Ho}), $\GG(\CC)$ has a maximal compact
subgroup $K$ and there is a homotopy equivalence $B\GG(\CC) \simeq BK $. Hence,
\[  BK^\wedge_\ell \simeq B\GG(\CC)^\wedge_\ell \simeq(B\bG)^\wedge_\ell \]
and as a consequence of \cite[Thms~7.4, 8.10, and 9.10]{BLO07} we obtain the
following.

\begin{prop}   \label{prop:compactalgebraic}
 With the above notation there is an isomorphism $ S'\cong S$ between a Sylow
 $\ell$-subgroup $S'$ of $K$ and a Sylow $\ell$-subgroup $S$ of $\bG$ inducing
 an isomorphism $\cF_{S'}(K)\cong \cF_S(\bG)$ of the corresponding fusion
 systems.
\end{prop}

\section{Proofs of Theorem~\ref{thm:compweights2} and Theorem~\ref{thm:compweights3}.}   \label{sec:proof}

We recall some notation from Section 7 of \cite{KMS}. Denote by $\cC$ the set of
finite sequences $\fc=(c_1,\ldots,c_t)$ of strictly positive integers including
the empty sequence $()$. For $\fc\in\cC$, write $|\fc|:=c_1+\cdots+c_t$ and let
$G_\fc:=\GL_{c_1}(\ell)\times\cdots\times\GL_{c_t}(\ell)$. Denote by $\cA(\fc)$
the set of irreducible characters of
$G_\fc=\GL_{c_1}(\ell)\times\cdots\times\GL_{c_t}(\ell)$ of the form
$\chi_1\cdots\chi_t$, where each $\chi_j$ is an extension to $\GL_{c_j}(\ell)$
of a Steinberg character of $\SL_{c_j}(\ell)$. Note that $|\cA(\fc)| =(\ell-1)^t$. 

We recall the description of $\ell$-stubborn subgroups in classical groups due
to Oliver \cite{O94}. As noted already by Oliver this is akin to the
Alperin--Fong type classification of $\ell$-radical subgroups of finite general
linear and symmetric groups in \cite{AF90}.

\begin{prop}   \label{prop:oliver}
 Suppose that $\bG$ is the compact connected Lie group $\U(n)$, $\Sp(n)$,
 $\SO(2n+1)$ or $\SO(2n)$, $n\geq 1$. Let $\ell$ be a prime, assumed odd unless
 $\bG=\U(n)$. If $\bG\ne\U(n)$, let $\cX$ denote the set of functions
 $$ f: \NN \times \cC \to \NN $$
 such that $\sum_{(\ga,\fc)} \ell^{|\ga| +|\fc|} f(\ga,\fc) = n $. If
 $\bG = \U(n)$, let $\cX$ denote the subset of the above set of functions $f$ 
 which additionally have the property $f(0,())\ne 2,4$ if $\ell=2$ and
 $f(0,()) \ne 3$ if $\ell=3$. 
 There is a bijection between the set of $\bG$-conjugacy classes of
 $\ell$-stubborn subgroups of $\bG$ and $\cX$ satisfying the following: Let
 $\bP \leq \bG$ be an $\ell$-stubborn subgroup whose class corresponds to
 $f\in\cX$.
 \begin{enumerate}[\rm(a)]
  \item If $\bG = \U(n)$, then
   \[ N_\bG(\bP)/\bP\cong \prod_{(\ga,\fc)}(\Sp_{2\ga}(\ell)\times G_\fc)\wr \fS_{f(\ga,\fc)}.\]
  \item If $\bG =\Sp(n)$ or $\bG =\SO(2n+1)$, then
   $$ N_\bG(\bP)/\bP\cong \prod_{(\ga,\fc)}(C_2\times\Sp_{2\ga}(\ell) \times G_\fc)\wr\fS_{f(\ga,\fc)}.$$
  \item If $\bG =\SO(2n)$, then $N_\bG(\bP)/\bP$ is isomorphic to the
   subgroup of
   $$\prod_{(\ga,\fc)}(C_2\times\Sp_{2\ga}(\ell)\times G_\fc)\wr\fS_{f(\ga,\fc)}$$
   consisting of those elements for which the number of non-trivial entries
   from the $C_2$ components is even.
 \end{enumerate}
\end{prop}

\begin{proof}
For $\fc\in\cC$ and a non-negative integer $\ga$ let $\bR_{\ga,\fc}$ denote the
irreducible subgroup $\Gamma_{\ell^\ga}^U\wr E_{\ell^{c_1}}\wr \cdots\wr
E_{\ell^{c_t}}$ of $\U(\ell^{\ga+|\fc|})$ defined in \cite[Def.~2]{O94}. We
regard $\bR_{\ga,\fc}$ as a subgroup of $\Sp(\ell^{\ga+|\fc|})$ (respectively
$\SO(2\ell^{\ga+|\fc|})$ via a standard embedding $\U(\ell^{\ga+|\fc|})\leq \Sp(\ell^{\ga+|\fc|})$ (respectively $\U(\ell^{\ga+|\fc|}) \leq \OO(2\ell^{\ga+|\fc|})$. By \cite[Thm~6]{O94},
\[N_{\U(\ell^{\ga+|\fc|})}(\bR_{\ga,\fc}) / \bR_{\ga,\fc} \cong \Sp_{2\ga}(\ell) \times G_\fc , \]
\[N_{\Sp(\ell^{\ga+|\fc|})}(\bR_{\ga,\fc}) / \bR_{\ga,\fc}\cong C_2 \times \Sp_{2\ga}(\ell) \times G_\fc, \]
and
\[N_{\OO(2\ell^{\ga+|\fc|})}(\bR_{\ga,\fc}) /\bR_{\ga,\fc} \cong C_2 \times \Sp_{2\ga}(\ell) \times G_\fc . \]
Moreover, from the proof of \cite[Thm~6]{O94} it can be checked that the full
inverse image of $\Sp_{2\ga}(\ell)\times G_\fc$ in
$N_{\OO(2\ell^{\ga+|\fc|})}(\bR_{\ga,\fc})$ lies in $\SO(2\ell^{\ga+|\fc|})$,
and that the generator of the $C_2$ factor lifts to an element of determinant
$-1$ in $\OO(2\ell^{\ga+|\fc|})$, specifically to an element of order $2$ with
$-1$ as an eigenvalue of multiplicity $\ell^{\ga+|\fc|}$.

If $\bR $ is a direct product of $m$ copies of $\bR_{\ga,\fc}$ diagonally
embedded via 
\[ \U(\ell^{\ga+|\fc|})\times\cdots\times\U(\ell^{\ga+|\fc|})
  \leq \U(m\ell^{\ga+|\fc|}), \]
then 
\[N_{\U(m\ell^{\ga+|\fc|}) }(\bR) /\bR \cong (N_{\U(\ell^{\ga+|\fc|})}(\bR_{\ga,\fc}) / \bR_{\ga,\fc}) \wr \fS_m \cong (\Sp_{2\ga}(\ell) \times G_\fc) \wr \fS_m\]
where $\fS_m$ acts by usual permutation of factors. Similarly, regarding $\bR$
as a diagonally embedded subgroup via
$\Sp(\ell^{\ga+|\fc|})^m\leq\Sp(m\ell^{\ga+|\fc|})$, then 
\[N_{\Sp(m\ell^{\ga+|\fc|}) }(\bR) /\bR \cong N_{\Sp(\ell^{\ga+|\fc|})}(\bR_{\ga,\fc}) / \bR_{\ga,\fc} \wr \fS_m\cong (C_2 \times \Sp_{2\ga}(\ell) \times G_\fc) \wr \fS_m\] 
and regarding $\bR$ as a diagonally embedded subgroup via
$\OO(2\ell^{\ga+|\fc|})^m \leq \OO(2m\ell^{\ga+|\fc|})$, then 
\[N_{\OO(2m\ell^{\ga+|\fc|}) }(\bR) /\bR \cong N_{\OO(2\ell^{\ga+|\fc|})}(\bR_{\ga,\fc}) / \bR_{\ga,\fc} \wr \fS_m \cong (C_2 \times \Sp_{2\ga}(\ell) \times G_\fc) \wr \fS_m.\] 
Note that since the space underlying $\bR_{\ga,\fc}$ in the orthogonal case is
even-dimensional, the elements of $\fS_m$ lift to elements of determinant $1$
in $\OO(2m\ell^{\ga+|\fc|})$.

Suppose that $\bG=\U(n)$. Then (a) follows from \cite[Thms~6 and~8]{O94} (see
also the last part of Definition~2 of \cite{O94}) with the bijection between
$\cX$ and the $\bG$-conjugacy classes of $\ell$-stubborn subgroups sending
$f\in\cX$ to the class of 
\[\bP = \prod_{(\ga,\fc)} \bR_{\ga,\fc}^{f(\ga,\fc)} \leq \prod_{(\ga,\fc)} \U(\ell^{\ga+|\fc|})^{f(\ga,\fc)} \leq \prod_{(\ga,\fc)} \U(f(\ga,\fc)\ell^{\ga+|\fc|}) \leq \U(n),\]
where the last inclusion is via a decomposition of the vector space underlying
$\U(n)$, and where 
\[N_\bG(\bP)/\bP \cong \prod_{(\ga,\fc)}N_{\U(f(\ga,\fc)\ell^{\ga+|\fc|})}(\bR_{\ga,\fc}^{f(\ga,\fc) })/\bR_{\ga,\fc}^{f(\ga,\fc) } \cong \prod_{(\ga,\fc)}(\Sp_{2\ga}(\ell)\times G_\fc) \wr \fS_{f(\ga,\fc)},\]
proving (a). The proof for the case that $\bG = \Sp(n)$ is entirely similar.
 
Suppose that $\bG= \SO(2n) $ and set $\hat \bG = \OO(2n)$. Since $\bG$ is the
connected component of $\hat \bG$ of index $2$, and $\ell $ is odd, the set of
$\ell$-stubborn subgroups of $\bG$ coincides with that of $\hat\bG$. 
By \cite[Thms~6 and~8]{O94}, the set of $\hat\bG$-conjugacy classes of
$\ell$-stubborn subgroups is in one to one correspondence with $\cX$ in the
same way as above where 
\[ \bP =\prod_{(\ga,\fc)} \bR_{\ga,\fc}^{f(\ga,\fc)} \leq \prod_{(\ga,\fc)}\U(\ell^{\ga+|\fc|})^{f(\ga,\fc)} \leq \prod_{(\ga,\fc)} \OO(2\ell^{\ga+|\fc|})^{f(\ga,\fc)} \leq \prod_{(\ga,\fc)} \OO(2f(\ga,\fc)\ell^{\ga+|\fc|}) \leq \OO(2n) \]
and 
\[ N_{\hat \bG}(\bP)/\bP \cong \prod_{(\ga,\fc)}(C_2 \times \Sp_{2\ga}(\ell) \times G_\fc)\wr \fS_{f(\ga,\fc)} \leq \prod_{(\ga,\fc)} \OO(2f(\ga,\fc)\ell^{\ga+|\fc|}) \leq \OO(2n). \] 
From the description above, it follows that $N_\bG(\bP)$ is the index $2$-subgroup of $N_{\hat \bG}(\bP)$ described in the statement of (c), and in particular, the $\hat \bG$-conjugacy class of $\bP$ is the same as the $\bG$-conjugacy class of $\bP$. This proves (c).
 
Finally suppose that $\bG = \SO(2n+1)$, and set $\hat\bG =\OO(2n+1)$. Then since
$\hat \bG =\bG\times\{\pm I\}$ and $\ell$ is odd, the $\bG$-conjugacy classes of
$\ell$-stubborn subgroups of $\bG$ are the same as the $\hat\bG$-conjugacy
classes of $\ell$-stubborn subgroups of $\hat \bG$. By \cite[Thms~6 and~8]{O94},
the set of $\hat \bG$-conjugacy classes of $\ell$-stubborn subgroups of
$\hat\bG$ are in one to one correspondence with $\cX$ in the same way as above
where now 
\[ \bP =\prod_{(\ga,\fc)} \bR_{\ga,\fc}^{f(\ga,\fc)} \leq \prod_{(\ga,\fc)}\U(\ell^{\ga+|\fc|})^{f(\ga,\fc)} \leq \prod_{(\ga,\fc)} \OO(2\ell^{\ga+|\fc|})^{f(\ga,\fc)} \leq \OO(2n)\times \OO(1) \leq \OO(2n+1) \]
and 
\[ N_{\hat \bG}(\bP)/\bP \cong \big(\prod_{(\ga,\fc)}(C_2 \times \Sp_{2\ga}(\ell) \times G_\fc)\wr \fS_{f(\ga,\fc)} \big) \times \OO(1)
 \leq \OO(2n) \times \OO(1) \leq \OO(2n+1). \]
Here, note that $\OO(1)\cong C_2$. Now the restriction to $N_\bG(\bP)/\bP$ of
the canonical surjection $N_{\hat\bG}(\bP)/\bP\to N_{\hat\bG}(\bP)/(\bP\times\OO(1))$ is an isomorphism and this proves (b).
\end{proof} 

\begin{thm}   \label{thm:compweights}
 Suppose that $\bG$ is the compact connected Lie group $\U(n)$, $\Sp(n)$, or
 $\SO(n)$, $n\geq 1$. Let $\ell$ be a prime, assumed odd unless $\bG= \U(n)$,
 $S\in\Syl_\ell(\bG)$ a Sylow $\ell$-subgroup of $\bG$, $\cF= \cF_S(\bG)$ the
 associated fusion system and $W$ the Weyl group of $\bG$. Then,
 $$ \bw(\cF) = |\Irr(W)|.$$
\end{thm} 

\begin{proof}
By Lemma~\ref{lem:crtostubborn} we have
\[\bw(\cF) = \sum_{\bP } z(N_\bG(\bP)/\bP) \]
where $\bP$ runs over representatives of $\bG$-conjugacy classes of
$\ell$-stubborn subgroups of $\bG$ and $z(N_\bG(\bP)/\bP)$ is the number of
irreducible characters of $N_\bG(\bP)/\bP$ of zero $\ell$-defect.

Suppose first that $\bG=\U(n)$. Let $\cX$ be as in
Proposition~\ref{prop:oliver},
let $f\in\cX$ and let $\bP$ be an $\ell$-stubborn of $\bG$ corresponding to $f$.
The only characters of zero $\ell$-defect of $\GL_c(\ell)$ or $\Sp_{2\ga}(\ell)$
are Steinberg characters and $\Sp_{2\ga}(\ell)$ has a unique Steinberg character
which we denote by $\chi_{\ga}$. Thus by Proposition~\ref{prop:oliver}(i) and
\cite[Prop.~9.3]{KMS}, the weights contributed by $\bP$ are in bijection with
the set
\[{\cW}_f:=\Big\{w :\bigcup_{(\ga,\fc)} \{\chi_\ga\} \times \cA(\fc) \to\{\ell\text{-cores}\} \quad
 \text{with}\sum_{\vhi\in\{\chi_\ga\}\times\cA(\fc)} |w(\vhi)| = f(\ga,\fc) \Big\}. \]
Letting $f$ range over $\cX$ and identifying $\{\chi_\ga\}\times\cA(\fc)$ with
$\cA(\fc)$, we see that the set of weights of $\cF$ is in bijection with the set
$\cW $ of assignments
\[ w: \NN \times\bigcup_{\fc\in\cC}\cA(\fc)\to \{\ell\text{-cores} \} \]
such that 
\[ \sum_\fc\sum_{(\ga,\vhi)\in\NN \times\cA(\fc)}\ell^{\ga +|c|}|w(\ga, \vhi)|=n. \] 
Here note that the restrictions on $\cX$ in Proposition~\ref{prop:oliver}(a) do
not play a role. By \cite[Lemma~7.15]{KMS}, applied with $e=r=1$ and by the
argument of the proof of Theorem~4.2 at the end of Section 7 of \cite{KMS} it
follows that the number of $\cF$-weights equals the number of ordinary
irreducible characters of $G(1,1,n)\cong \fS_n$, the Weyl group of $\U(n)$.

Suppose next that $\bG = \Sp(n)$ or $\SO(2n+1)$. Let $\cX$ be as in
Proposition~\ref{prop:oliver}, let $f\in\cX$ and let $\bP$ be an $\ell$-stubborn
subgroup of $\bG$ corresponding to $f$. Since $\ell$ is odd, all irreducible
characters of $C_2$ are of $\ell$-defect~$0$, we obtain from
Proposition~\ref{prop:oliver}(b) and by \cite[Prop.~9.3]{KMS} that the weights
contributed by $\bP$ are in bijection with the set
\[\tilde \cW_f:=\Big\{w :\bigcup_{(\ga,\fc)}\{\chi_{\ga}\}\times C_2\times \cA(\fc) \to\{\ell\text{-cores}\} \quad
 \text{with}\sum_{\vhi\in \{\chi_\ga\} \times C_2 \times \cA(\fc) } | w(\vhi) | = f(\ga,\fc) \Big\}. \]
Letting $f$ range over $\cX$ and identifying $\{\chi_\ga\}\times C_2 \times
 \cA(\fc)$ with $C_2\times \cA(\fc)$, it follows that the set of $\cF$-weights
is in bijection with the set $\tilde \cW $ of assignments 
\[ w: \NN\times C_2\times\bigcup_{\fc\in\cC}\cA(\fc)\to \{\ell\text{-cores}\} \]
such that
\[ \sum_\fc\sum_{(\ga,x,\vhi)\in\NN\times C_2\times\cA(\fc)}\ell^{\ga +|c|} |w(\ga, x, \vhi)|=n. \] 
As in the previous case by \cite[Lemma~7.15]{KMS} now applied with $e=2$, $r=1$
and with $\tilde \cW$ in place of $\cW$ and by the argument of the proof of
Theorem~4.2 of \cite{KMS} we obtain that the number of $\cF$-weights equals the
number of ordinary irreducible characters of $G(2,1,n)\cong W(C_n)=W(B_n)$.

Now suppose that $\bG=\SO(2n)$ and let $f$, $\tilde \cW_f$ be as just defined
above. Let $C_2$ act on $\tilde \cW_f$ via
$y.w(\ga,x,\vhi):=w(\ga,y +x,\vhi)$ for $y\in C_2$, $w\in\tilde \cW_f$ and
$(\ga,x,\vhi)\in \NN\times C_2 \times\cA$. By Proposition~\ref{prop:oliver}(c)
and \cite[Prop.~9.3]{KMS}, applied with $N_{\OO(2n)}(\bP) /\bP$ in place of $G$
and $N_{\SO(2n)}(\bP) /\bP$ in place of $M$, the weights contributed by $\bP$
are indexed by $\ZZ/ 2\ZZ$-orbits of $\tilde \cW_f$ with each orbit contributing
as many weights as the order of the stabiliser of a point of the orbit. Again,
letting $f$ range over all elements of $\cX$, we obtain that the $\cF$-weights
are indexed by the $\ZZ/2\ZZ$-orbits of $\cW$ with each orbit contributing as
many weights as the order of the stabiliser of a point of the orbit, where $C_2$
acts on $\tilde \cW_f$ as above, that is via
$y.w(\ga,x,\vhi):=w(\ga,y +x,\vhi)$ for $y\in C_2$, $w\in\tilde \cW_f$ and
$(\ga,x,\vhi)\in \NN\times C_2 \times\cA$. 
As before, by \cite[Lemma~7.15]{KMS} now applied with $e=2 $, $r=2$ and with
$\tilde \cW$ in place of $\cW$ and by the argument of the proof of Theorem~4.2
of \cite{KMS} we obtain that the number of $\cF$-weights equals the number of
ordinary irreducible characters of $G(2,2,n) \cong W(D_n)$.
\end{proof} 

\begin{proof}[Proof of Theorem~\ref{thm:compweights2}]
We use again Lemma~\ref{lem:crtostubborn}.
By \cite[Lemma~1.5(ii)]{JMO92-I}, any $\ell$-stubborn subgroup of $\bG$
contains $Z(\bG)$ and a subgroup $\bP\leq \bG$ which contains $Z(\bG)$ is
$\ell$-stubborn if and only if $\bar\bP:=\bP/Z(\bG)$ is an $\ell$-stubborn
subgroup of $\bar\bG:=\bG/Z(\bG)$. For any $\bP\leq \bG$ containing $Z(\bG)$,
$N_\bG(\bP)/\bP\cong N_{\bar\bG}(\bar\bP)/\bar\bP$ and the Weyl groups of
$\bG$ and $\bar\bG$ are isomorphic. 
Thus, we may assume that $Z(\bG)=1 $ and hence that $\bG$ is a direct product of
simple compact Lie groups. By \cite[Lemma~1.5(i)]{JMO92-I}, the $\ell$-stubborn
subgroups of a direct product of compact connected Lie groups are the direct
products of $\ell$-stubborn subgroups of the factors and the decomposition into
direct factors respects conjugacy, normalisers, and Weyl groups hence we may
assume that $\bG$ is simple. 

Assume first that $\bG$ is of classical type, so $\ell\ge3$. Applying the above
arguments again, we may assume that $\bG$ is one of the groups $\U(n)$,
$\Sp(n)$, $\SO(2n+1)$ or $\SO(2n)$ and we are done by
Theorem~\ref{thm:compweights}. Now assume that $\bG$ is an exceptional group.
Suppose first that the Weyl group of $\bG$ is an $\ell'$-group. Then every
$\ell$-toral subgroup of $\bG$ is contained in a maximal torus of $\bG$, and by
Lemma~\ref{lem:radicalistubborn}(a)
the only $\ell$-stubborn subgroups of $\bG$ are the maximal tori and the result
is immediate. Thus the only cases left are $\bG=E_6$ with $\ell=5$, $\bG=E_7$
with $\ell=5,7$ and $\bG=E_8$ with $\ell=7$. These are handled by the next
proposition. Note that in all of these cases Sylow $\ell$-subgroups of $W$ are
cyclic.
\end{proof}

\begin{prop}   \label{prop:En}
 Suppose $\bG =E_n$ for $n\in\{6,7,8\}$ and let $\ell$ be a good prime
 dividing $|W(\bG)|$. Then $\bw(\cF_\ell(\bG)) = |\Irr(W)|$ where
 $\cF_\ell(\bG)$ denotes the $\ell$-fusion system of $\bG$.
\end{prop}

\begin{proof}
The fusion system $\cF=\cF_\ell(\bG)$ is calculated in \cite[Thm~B(c)]{OR19}
and its structure is described in \cite[Tab.~5.1]{OR19}. Apart from $\{S\}$ and
$\{T\}$ there is one further class $\mathcal{B}=E^\cF$ of $\cF$-centric radical
subgroups with $E \cong \ell_+^{1+2}$. Now $\Out_\cF(T)=W(\bG)$, and
\cite[Thm~4.5]{S23} shows that this determines $\Out_\cF(P)$ for the remaining
$P\in\cF^{cr}$. We deduce that $\Out_\cF(P)$ and $z=z(\Out_\cF(P))$ are given by
the following table:

\begin{table}[ht]
$\begin{array}{c|cccccccc|c}
 (n,\ell) & S& z&& T& z&& E& z& |\Irr(W(\bG))|\\ \hline
 (6,5) & C_{4} \times C_2& \bf{8}&& W& \bf{15}&& \SL_2(5).2& \bf{2}& 25\\
 (7,5) & C_4 \times C_2 \times \fS_3& \bf{24}&& W& \bf{30}&& \SL_2(5).2 \times \fS_3& \bf{6}& 60\\
 (7,7) & C_6 \times C_2& \bf{12}&& W& \bf{46}&& \SL_2(7).2 & \bf{2}& 60\\
 (8,7) & C_6 \times (C_2)^2 & \bf{24}&& W& \bf{84}&& \SL_2(7).2 \times C_2& \bf{4}& 112\\
\end{array}$
\end{table}
The equality ensues.
\end{proof}

\begin{proof}[Proof of Theorem~\ref{thm:compweights3}]
This is immediate from Proposition~\ref{prop:compactalgebraic} and
Theorem~\ref{thm:compweights2}.
\end{proof}

\section{Bad primes} \label{sec:badprimes}

\begin{prop}   \label{prop:badfail}
 Suppose that $\bG=\Sp(n)$ with $n\ge2$ and $\ell=2$. Then $\bw(\cF) < |\Irr(W)|$.
\end{prop} 

\begin{proof}
For $\ga\in\NN$ and $\fc\in\cF$, let $\bS_{\ga,\fc}$ (respectively
$\bar\bS_{\ga,\fc}$) denote the irreducible subgroup
$\Gamma_{2^\ga}^\Sp\wr E_{2^{c_1}}\wr\cdots\wr E_{2^{c_t}}$ (respectively
$\bar\Gamma_{2^\ga}^\Sp\wr E_{2^{c_1}}\wr\cdots \wr E_{2^{c_t}}$) of
$\Sp(2^{\ga+|\fc|})$ defined in \cite[Def.~2]{O94}. By \cite[Thms~6 and~8]{O94},
the $\bG$-classes of $2$-stubborn subgroups of $\bG$ are in
bijection with the set of ordered pairs of functions $(f, f')$ where 
\[ f: \NN\times\cC\to\NN\quad \text{and}\quad f': \NN\times\cC\to\NN \]
are such that
\[\sum_{(\ga,\fc)}2^{\ga +|\fc|} f({\ga,\fc})
   + \sum_{(\ga', \fc')}2^{\ga'+|\fc'|} f'({\ga',\fc'}) =n \]
and $f'((0,()))\ne 2,4$. If $\bP\leq\bG$ is a $2$-stubborn subgroup
corresponding to $(f,f')$ then
\[\bP \cong \prod_{(\ga,\fc)} \bS_{\ga,\fc}^{f(\ga,\fc)}
  \times \prod_{(\ga',\fc')} \bar \bS_{\ga',\fc'}^{f(\ga',\fc')} \]
and 
\[ N_\bG(\bP)/\bP\cong \prod_{(\ga,\fc)}(\GO_{2\ga+2}^-(2)\times G_\fc) \wr \fS_{f({\ga, \fc) }} \times \prod_{(\ga',\fc')}(\Sp_{2\ga'}(2) \times G_{\fc'})\wr\fS_{f'(\ga', \fc')} .\]

Now $\GO_{2}^-(2)\cong\fS_3$ has a unique irreducible character of $2$-defect
$0$ while $\GO_{2\ga+2}^-(2)$ has no irreducible characters of $2$-defect $0$
for $\ga\geq 1$ so the base subgroup $\GO_{2\ga+2}^-(2)\times G_\fc$ has
a unique irreducible character of $2$-defect $0$ if $\ga =0$ and none if
$\ga\geq 1$. The base subgroup $\Sp_{2\ga'}(2)\times G_{\fc'}$ has a
unique irreducible character of $2$-defect $0$ for all $\ga'\ge0$. Further,
the symmetric group
$\fS_m$ has one block of 2-defect $0$ if $m$ is a triangular number and none
otherwise. From this and \cite[Lemma 8.4]{BLO07}, it follows that the weight
contribution of $\bP$ is $1$ if $\ga=0$ and the values of $f$ and $f'$ are
triangular numbers (including 0), and is $0$ otherwise. So, the number of
$\cF$-weights equals the number of pairs $(f,f')$ of functions
\[ f:\cC\to\NN\quad \text{and}\quad f': \NN\times\cC\to\NN \]
such that 
 \[\sum_\fc 2^{|\fc|}f(\fc)
  +\sum_{(\ga',\fc')}2^{\ga'+|\fc'|}f'({\ga',\fc'}) =n \] 
and such that all values of $f$ and $f'$ are triangular numbers. On the other,
hand using the analysis for the odd $\ell$ case and noting that when $\ell=2$,
$\cA(\fc)$ is a singleton and that the set of $2$-cores is in bijection with the
set of triangular numbers, we have that the number of bipartitions of $n$, i.e.,
$|\Irr(W)|$, equals the number of pairs $(f,f')$ of functions
\[ f:\NN \times \cC\to\NN\quad \text{and}\quad f': \NN\times\cC\to\NN \]
such that 
\[\sum_{(\ga,\fc)} 2^{|\ga|+ |\fc|}f(\ga,\fc)
  +\sum_{(\ga',\fc')}2^{\ga'+|\fc'|}f'({\ga',\fc'}) =n\]
and such that all values of $f$ and $f'$ are triangular numbers.
Thus, the number of weights is strictly less than $|\Irr(W)|$ for all $n\geq 2$.
\end{proof}

\begin{exmp}
For $n=2$, we get 4 weights for $\Sp(2)$, namely for $((1);(0,()))$ with
$f=1,f'=0$, $(();(0,()))$ with $f=f'=1$, $(();(1,()))$ and $(();(0,(1)))$ with
$f=0,f'=1$, and so by the proof of Theorem~\ref{thm:compweights3},
$4=\bw(\cF) < |\Irr(W(\bG))|=5$, where $\bG$ is a simple algebraic group over
$\overline{\FF}_p$ of type $C_2$ and $\cF$ is its $2$-fusion system. Note,
however, that if $\cF$ is the $2$-fusion system of $\CSp_4(q)$ with
$\nu_2(q-1)> 2$ there are five $\cF$-centric radical subgroups each contributing
a single weight (so $\bw(\cF)=5$), but two of these subgroups are fused in
$\CSp_4(q^2)$.
\end{exmp}

\begin{exmp}   \label{exmp:finite}
 Let $\bG$ be a connected reductive group of symplectic or odd dimensional
 orthogonal type over a field of odd characteristic and $F:\bG\to\bG$ a
 Frobenius map. Then the number $\bw(\cF)$ of 2-weights for the principal
 2-block of $\bG^F$ equals the number of unipotent conjugacy classes of $\bG^F$
 (see \cite[Prop.~(2A)]{AC95} in general). Now the Springer correspondence
 defines an injective map from $\Irr(W)$ to the set of unipotent classes of
 $\bG^F$, so we conclude $\bw(\cF)\ge|\Irr(W)|$. In fact, since the Springer
 correspondence is not surjective in general we have $\bw(\cF)>|\Irr(W)|$ for
 large enough rank. E.g., for $\bG$ of type $B_4$ or type $C_6$ we have
 $\bw(\cF)=|\Irr(W)|+1$.
\end{exmp}

\begin{prop}   \label{prop:G2}
 Suppose $\bG =G_2$. Then $\bw(\cF) = |\Irr(W)| =6$ for all primes $\ell$.
\end{prop}

\begin{proof}
If $\ell>3$, then we are done by Theorem~\ref{thm:compweights}. The $2$-stubborn
and $3$-stubborn subgroups of $\bG$ are determined in \cite[Lemma~3.2]{JMO95}.
There are six classes of $2$-stubborn subgroups $\bP$, with $N(\bP)/\bP$
isomorphic to $1$, $\fS_3$ (for three of the classes), $\fS_3\times\fS_3$, and
$\GL_3(2)$, respectively. Each of these contributes one weight. There are two
classes of $3$-stubborn subgroups $\bP$ with $N(\bP)/\bP$ isomorphic to
$C_2\times C_2$ and $\GL_2(3)$, respectively. The first contributes $4$ weights
and the second contributes $2$ weights.
\end{proof}

\begin{prop}   \label{prop:F4}
 Suppose $\bG =F_4$. Then $\bw(\cF)= 22 < |\Irr(W)|=25$ for $\ell=3$.
\end{prop}

\begin{proof}
The $3$-stubborn subgroups of $\bG$ are determined in
\cite[Prop.~3.6]{Vi01}: there are seven conjugacy classes, with respective
automisers $D_8$, $(C_2\times\Sp_2(3)).2$ (twice), $\Sp_2(3)\wr2$, $\GL_2(3)$,
$\SL_3(3)$ and $W(F_4)$. The respective number of weights for those is 5, 4
(twice), 2, 2, 1, 4, adding up to~22.
\end{proof}


\end{document}